\documentclass[12pt,a4paper]{amsart}
\usepackage[latin1]{inputenc}
\usepackage{graphicx}
\usepackage{amssymb, amsmath}
\usepackage{geometry}
\usepackage{enumerate}
\usepackage[colorlinks=true,linkcolor=blue,urlcolor=blue,citecolor=blue]{hyperref}
\newtheorem{theorem}{Theorem}[section]
\newtheorem{definition}[theorem]{Definition}
\newtheorem{lemma}[theorem]{Lemma}
\newtheorem{corollary}[theorem]{Corollary}
\newtheorem{proposition}[theorem]{Proposition}

\theoremstyle{definition}

\newtheorem{remark}[theorem]{Remark}

\newcommand{\N}{\mathbb{N}}

\begin{document}
\title[Doubling constant in metric measure spaces]{The least doubling constant of a metric measure space}

\author{Javier Soria}
\address{Department of Mathematics and Computer Science, University of Barcelona, Gran Via 585, E-08007 Barcelona, Spain.}
\email{soria@ub.edu}

\author{Pedro Tradacete}
\address{Instituto de Ciencias Matem\'aticas (CSIC-UAM-UC3M-UCM)\\
Consejo Superior de Investigaciones Cient\'ificas\\
C/ Nicol\'as Cabrera, 13--15, Campus de Cantoblanco UAM\\
28049 Madrid, Spain.}
\email{pedro.tradacete@icmat.es}

\thanks{The first author has  been partially supported by the Spanish Government grant MTM2016-75196-P (MINECO/FEDER, UE) and the Catalan Autonomous Government grant 2017SGR358.
The second author gratefully acknowledges support of Spanish Ministerio de Econom\'{\i}a, Industria y Competitividad through grants MTM2016-76808-P, MTM2016-75196-P and the ``Severo Ochoa Programme for Centres of Excellence in R\&D'' (SEV-2015-0554). The second author wishes to thank Beno\^it Kloeckner (Paris-Est) and Estibalitz Durand-Cartagena (UNED) for interesting discussions on the topic of this paper.}


\begin{abstract}
We study the least doubling constant $C_{(X,d)}$, among all doubling measures $\mu$ supported on a metric space $(X,d)$. In particular, we prove that for every metric space with more than one point,  $C_{(X,d)}\ge 2$. We also describe some further properties of $C_{(X,d)}$ and compute its value for several important examples.
\end{abstract}

\maketitle

\thispagestyle{empty}

\section{Introduction and motivation}
Given a metric space $(X,d)$, a Borel regular measure $\mu$ on $X$ is called doubling if there exists a constant $C\geq1$ such that, for every $x\in X$ and $r>0$,
\begin{equation}\label{doublingC}
\mu(B(x,2r))\leq C\mu(B(x,r)),
\end{equation}
where $B(x,r)=\{y\in X: d(x,y)<r\}$. If this is the case, the metric measure space $(X,d,\mu)$ will be called a space of homogeneous type (cf. \cite{CW}). Given such $(X,d,\mu)$, we will denote by $C_\mu$ the best possible constant appearing in \eqref{doublingC}; that is,
\begin{equation*}\label{bestconstant}
C_\mu=\sup_{x\in X, r>0}\frac{\mu(B(x,2r))}{\mu(B(x,r))}.
\end{equation*}

For convenience, let us introduce the following definition.

\begin{definition}\label{optconstxd}
Given a metric space $(X,d)$, we define the least doubling constant as
$$
C_{(X,d)}=\inf\big\{C_\mu:\mu \textrm{ doubling measure on }(X,d)\big\}.
$$
\end{definition}

If no doubling measure exists in $(X,d)$, we will write that $C_{(X,d)}=\infty$. We would like to note that all references we have found in the literature, place the constant $C_{(X,d)}$ in the interval $[1,\infty)$. One can easily check, unless the metric space reduces to a singleton, that $C_{(X,d)}>1$. An elementary argument shows that the lower bound $C_{(X,d)}\geq\varphi=\frac{1+\sqrt{5}}{2}$ always holds (Proposition~\ref{goldenratio}). However, with some more work it will be shown that, in fact, this estimate can be improved to $C_{(X,d)}\geq2$ (Theorem~\ref{Cgeq2}). 

In general, it is not true that on every metric space $(X,d)$ one can always find such a doubling measure (e.g., $X=\mathbb Q$ with the standard euclidean distance; see also \cite{S}). However, if a metric space $(X,d)$ supports a non-trivial doubling measure, then there exists $K\in\mathbb N$ such that, for every $x\in X$ and $r>0$, the number of $r$-separated points in $B(x,2r)$ is bounded by $K$, where two points $x,y\in X$ are $r$-separated provided $d(x,y)\geq r$ (cf. \cite{CW} and Proposition \ref{p:lowerboundsr} below). If this property holds, then we say that $(X,d)$ is a doubling space (hence, a space of homogeneous type is doubling). Conversely, for compact \cite{VK}, or more generally, complete metric spaces \cite{LS}, being doubling (in the metric sense) implies the existence of a doubling measure (see also \cite{KRS, Wu} for related developments).

The doubling constant $C_\mu$ given above should not be confused with the doubling constant of a metric space $(X,d)$, which is usually referred to as the minimal $k\in\mathbb N$ such that every ball $B(x,r)$ can be covered by at most $k$ balls of radius $r/2$. This leads to the definition of doubling dimension of $(X,d)$ as $K_{(X,d)}=\lceil \log_2(k)\rceil$, which is of significance in metric embedding theory (cf. \cite{ABN,A}). 

To motivate our goal, let us see what happens for a couple of particular, but significant, examples.  For $\alpha>-1$, consider the locally integrable measure $d\mu_\alpha(x)=|x|^\alpha\,dx$.  It is easy to see that $\mu_\alpha$ is doubling in $\mathbb R$ with the euclidean distance (this also follows from the fact that $|x|^\alpha$ is a weight in the Muckenhoupt class $A_\infty$ \cite{GCRF}). For the interval $I=(-1,1)$ we obtain that
 $$
 C_{\mu_\alpha}\ge\frac{\mu_\alpha(-2,2)}{\mu_\alpha(-1,1)} =2^{\alpha+1},
 $$
 and with $I=(1,3)$,
 $$
 C_{\mu_\alpha}\ge\frac{\mu_\alpha(0,4)}{\mu_\alpha(1,3)} =\frac{4^{\alpha+1}}{3^{\alpha+1}-1}.
 $$
 Hence,
 $$
 C_{\mu_\alpha}\ge\max\bigg\{2^{\alpha+1},\frac{4^{\alpha+1}}{3^{\alpha+1}-1}\bigg\}\ge 2,
 $$
 with equality $C_{\mu_\alpha}=2$ only when $\alpha=0$; i.e., for the Lebesgue measure. A second example, this time in the most trivial discrete setting, comes when we take a set $X=\{1,2\}$ with 2 points, any measure $\mu$ and any distance $d$. Then,
 $$
 C_\mu=\max\bigg\{1+\frac{\mu(\{1\})}{\mu(\{2\})},1+\frac{\mu(\{2\})}{\mu(\{1\})}\bigg\}\ge2.
 $$

 \medskip
 In the rest of this paper, we will start by giving, in Section~\ref{sec2}, some preliminary results which are of interest in this context. We also show some important properties for spaces of homogeneous type, like the fact that there is no upper bound for $C_\mu$ (Proposition~\ref{infinity}). In particular, we will show a short proof that $C_{(X,d)}\geq\varphi=\frac{1+\sqrt{5}}{2}$ (Proposition~\ref{goldenratio}), as long as $X$ contains more than one point. In Section~\ref{sec3}, we prove that actually $C_{(X,d)}\geq2$. Section \ref{sec4} is devoted to more general lower bounds for $C_{(X,d)}$, and finally, Section~\ref{sec5} provides explicit values of $C_{(X,d)}$ for specific metric spaces, both in the continuous and the discrete settings (see also \cite{LWW} for further considerations).
\medskip

The explicit value of the doubling constant $C_\mu$ is of significance in several recent developments in the theory of metric measure spaces (see for instance \cite{AB, BB, KKST}). We refer the reader to the monographs \cite{heinonen, HKST} for background and applications of this theory.

\section{Preliminary results and a universal lower bound for $C_\mu$}\label{sec2}

Throughout, we will always assume that $(X,d)$ is a metric space on which doubling measures exist, and that $X$ contains at least 2 points. Moreover, all balls $B(x,r)=\big\{y\in X:d(x,y)<r\big\}$ on $(X,d)$ are open sets 
and we only consider non-trivial measures $\mu$, in the sense that $0<\mu(B(x,r))<\infty$, for every $x\in X$ and $r>0$. Also, for $r\in \mathbb R$, as usual we denote 
$$
\lceil r\rceil=\min\{n\in\mathbb Z: r\leq n\}\quad\quad \text{and}\quad\quad \lfloor r \rfloor= \max\{n\in\mathbb Z: n\leq r\}.
$$

\begin{lemma}\label{lrs}
Let $(X,d,\mu)$ be a space of homogeneous type. For every $x\in X$ and $0<s<r$, we have
$$
\mu(B(x,r))\leq C_\mu^{\lceil\log_2(\frac{r}{s})\rceil} \mu(B(x,s)).
$$
\end{lemma}

\begin{proof}
Let $n=\lceil\log_2(\frac{r}{s})\rceil$. This means that
$$2^{n-1}s< r\leq 2^n s.$$
Hence, iterating  \eqref{doublingC} it follows that
$$
\mu(B(x,r))\leq \mu(B(x,2^ns))\leq C_\mu \mu(B(x,2^{n-1}s))\leq\cdots\leq C_\mu^{\lceil\log_2(\frac{r}{s})\rceil} \mu(B(x,s)).
$$
\end{proof}

The following result is a quantitative version of \cite[Remarque, p.\ 67]{CW}:

\begin{proposition}\label{p:lowerboundsr}
Let $(X,d,\mu)$ be a  space of homogeneous type. If there exist $x\in X$ and $(y_j)_{j=1}^N\subset B(x,r)$ such that $B(y_j,s)\subset B(x,r)$, for some $0<s\leq 2r$, and $B(y_j,s)\cap B(y_k,s)=\emptyset$, for every $j,k\in\{1,\ldots, N\}$, $j\neq k$, then
$$
N\leq C_\mu^{\lceil\log_2(\frac{2r}{s})\rceil}.
$$
\end{proposition}

\begin{proof}
Let $j_0\in\{1,\ldots,N \}$ be such that 
$$
\mu(B(y_{j_0},s))=\min_{j\in\{1,\ldots,N\}}\mu(B(y_j,s)).
$$
Since $B(x,r)\subset B(y_{j_0},2r)$, it follows that
\begin{eqnarray*}
N\mu(B(y_{j_0},s))&\leq &\sum_{j=1}^N\mu(B(y_j,s))\leq\mu(B(x,r))\\
&\leq&\mu(B(y_{j_0},2r))\leq C_\mu^{\lceil\log_2(\frac{2r}{s})\rceil}\mu(B(y_{j_0},s)).
\end{eqnarray*}
\end{proof}

The previous result yields the well-known fact that balls in spaces of homogeneous type are totally bounded. For completeness, we include a short proof:

\begin{corollary}\label{totallybounded}
If  $(X,d,\mu)$ is a  space of homogeneous type, then every ball in $X$ is totally bounded.
\end{corollary}

\begin{proof}
Let $x\in X$ and $r>0$. Take any $\varepsilon\in(0,2r)$. We can inductively construct a sequence of points in $B(x,r)$ which are $\varepsilon$-separated: let $x_0=x$, and if $B(x,r)\not\subset B(x,\varepsilon)$, let $x_1\in B(x,r)\backslash B(x,\varepsilon)$. Similarly, if $B(x,r)\not\subset B(x_0,\varepsilon)\cup B(x_1,\varepsilon)$, then let $x_2\in B(x,r)\backslash(B(x_0,\varepsilon)\cup B(x_1,\varepsilon))$. Following in this way, for each $m\in \mathbb N$ either $B(x,r)\subset \bigcup_{i=0}^m B(x_i,\varepsilon)$ or we can pick $x_{m+1}\in B(x,r)\backslash \bigcup_{i=0}^m B(x_i,\varepsilon)$, and keep going.

Since we clearly have that 
$$
B(x_i,\varepsilon/2)\cap B(x_j,\varepsilon/2)=\emptyset,
$$ 
for $i\neq j$, and 
$$
B(x_i,\varepsilon/2)\subset B(x,r+\varepsilon/2),
$$ 
for every $i$, Proposition \ref{p:lowerboundsr} yields that for some $m\leq C_\mu^{1+\lceil \log_2(1+2r/\varepsilon)\rceil}$ we must have 
$$
B(x,r)\subset \bigcup_{i=0}^m B(x_i,\varepsilon).
$$ 
Thus, $B(x,r)$ is totally bounded, as claimed.
\end{proof}

We observe next that if we replace, in Definition~\ref{optconstxd},  the infimum by the supremum, then we get no interesting information. In fact, we can prove the following:

\begin{proposition}\label{infinity} If $(X,d)$ is a metric space, containing at least 2 points, then
$$
\sup\big\{C_\mu:\mu \textrm{ doubling measure on }(X,d)\big\}=\infty.
$$
\end{proposition}

\begin{proof}
Given a fixed doubling measure $\mu$ on $(X,d)$, we set $\varepsilon=\frac{C_\mu-1}{2C_\mu}>0$. Let $x\in X$ and $r>0$ such that
\begin{equation}\label{munu}
1<\frac{C_\mu+1}2=C_\mu(1-\varepsilon)<\frac{\mu(B(x,2r))}{\mu(B(x,r))}=1+\frac{\mu(B(x,2r)\setminus B(x,r))}{\mu(B(x,r))}.
\end{equation}
For $n\in\mathbb N$, let us define the function
$$
f_n(y)=\begin{cases}
   n,   & \text{ if } y\in B(x,2r)\setminus B(x,r), \\
    1,  & \text{otherwise},
\end{cases}
$$
and the measure $d\nu_n=f_nd\mu$. Since $f_n$ and $1/f_n$ are bounded functions, we have that $\nu_n$ is a doubling measure in $(X,d)$. Moreover, using \eqref{munu}
\begin{align*}
C_{\nu_n}&\ge\frac{\nu_n(B(x,2r))}{\nu_n(B(x,r))}=1+\frac{\nu_n(B(x,2r)\setminus B(x,r))}{\nu_n(B(x,r))}\\
&=1+n\,\frac{\mu(B(x,2r)\setminus B(x,r))}{\mu(B(x,r))}>n\frac{C_\mu-1}2+1\underset{n\to\infty}{\longrightarrow}\infty.
\end{align*}
\end{proof}

A point $x$ in a metric space $(X,d)$ is isolated if there is some $r>0$ such that $B(x,r)=\{x\}$. It is well-known that if $(X,d,\mu)$ is a space of homogeneous type, then $\mu(\{x\})>0$ if and only if $x$ is isolated \cite[Lemma 2]{KW}. For completeness, we are going to give another proof of this fact based on a reverse inequality.

\begin{proposition}
Let $(X,d)$ be a metric space with non-isolated points; that is,  the set $A=\{x\in X: x \text{ is non-isolated}\}\neq\emptyset$. If $\mu$ is a doubling measure in $(X,d)$, with constant $C_\mu$, then for every $x\in A$, there exists a decreasing sequence $r_j(x)\downarrow 0$, $j\to\infty$, such that
\begin{equation}\label{reverse}
\mu\big(B(x,4r_j(x))\big)\ge\big(1+C^{-2}_\mu\big)\mu\big(B(x,r_j(x))\big).
\end{equation}
In particular, $\mu(\{x\})=0$, for every $x\in A$.
\end{proposition}

\begin{proof}
Let $x\in A$ and pick $x_1\in X\setminus\{x\}$. Set $r_1=d(x,x_1)/2>0$. Then it is easy to see that
\begin{itemize}
\item[(i)] $B(x_1,r_1)\cap B(x,r_1)=\emptyset,$
\item[(ii)] $B(x_1,r_1)\subset B(x,4r_1),$ and
\item[(iii)] $B(x,r_1)\subset B(x_1,4r_1).$
\end{itemize}
Therefore,
\begin{align*}
\mu\big(B(x,4r_1)\big)&\ge \mu\big(B(x_1,r_1)\big)+\mu\big(B(x,r_1)\big)\ge \frac1{C^{2}_\mu}\mu\big(B(x_1,4r_1)\big)+\mu\big(B(x,r_1)\big)\\
&\ge \big(1+C^{-2}_\mu\big)\mu\big(B(x,r_1)\big).
\end{align*}
We now choose $x_2\in B(x,r_1)\setminus\{x\}$, and define $r_2=d(x,x_2)/2<r_1/2<r_1$. As before, we have
\begin{itemize}
\item[(i)] $B(x_2,r_2)\cap B(x,r_2)=\emptyset,$
\item[(ii)] $B(x_2,r_2)\subset B(x,4r_2),$ and
\item[(iii)] $B(x,r_2)\subset B(x_2,4r_2).$
\end{itemize}
Thus, we also get
$$
\mu\big(B(x,4r_2)\big)\ge  \big(1+C^{-2}_\mu\big)\mu\big(B(x,r_2)\big).
$$
Iterating this process, we obtain $r_j<r_{j-1}/2<r_1/2^{j-1}$, $j=2,3,\dots$, so that $r_j\downarrow0$ and \eqref{reverse} holds. 

Finally, if $x\in A$, then
$$
\mu(\{x\})=\mu\bigg(\bigcap_{j\in\mathbb N}B(x,4r_j(x))\bigg)=\lim_{j\to\infty}\mu(B(x,4r_j(x)))\ge \big(1+C^{-2}_\mu\big)\mu(\{x\}),
$$
which proves that $\mu(\{x\})=0$.
\end{proof}

We finish this section with an easy proof of a weaker lower bound than the one given in Theorem~\ref{Cgeq2}, but still of interest (for some time, it was the best known estimate):

\begin{proposition}\label{goldenratio}
If $(X,d)$ is any metric space with $|X|>1$, then $C_{(X,d)}\geq \varphi=\frac{1+\sqrt{5}}{2}$.
\end{proposition}

\begin{proof}
Pick $x,y\in X$ and set $r=d(x,y)>0$. Take any $\lambda>0$. Suppose first 
\begin{equation}\label{xy}
\mu(B(x, {2r}/{3}))\leq\lambda\mu(B(y, {r}/{3})).
\end{equation}
In this case, since $B(x, {2r}/{3})\cap B(y, {r}/{3})=\emptyset$ and
$$
B(x, {2r}/{3})\cup B(y, {r}/{3})\subset B(x, {4r}/{3}),
$$
it follows that
$$
C_\mu\mu(B(x, {2r}/{3}))\geq \mu(B(x, {4r}/{3}))\geq (1+ 1/\lambda)\mu(B(x, {2r}/{3})).
$$
Thus, in this case, we have $C_\mu\geq 1+1/\lambda$.

Similarly, if
\begin{equation}\label{yx}
\mu(B(y, {2r}/{3}))\leq\lambda\mu(B(x, {r}/{3})),
\end{equation}
then one also gets $C_\mu\geq1+1/\lambda$. Finally, let us assume that neither   \eqref{xy} nor \eqref{yx} hold. In that case we have
\begin{align*}
C_\mu(\mu(B(x, {r}/{3}))+\mu(B(y, {r}/{3})))&\geq \mu(B(x, {2r}/{3}))+\mu(B(y, {2r}/{3}))\\
&>\lambda (\mu(B(y, {r}/{3}))+\mu(B(x, {r}/{3}))),
\end{align*}
which gives $C_\mu>\lambda$.

Since this works for any $\lambda>0$, in any case we get 
$$
C_\mu \geq \sup_{\lambda>0}\min\{\lambda,1+{1}/{\lambda}\}.
$$
Optimizing in $\lambda>0$, the result follows.
\end{proof}

It should be noted that, despite the apparently irrelevant  choice of radii of the balls considered in the previous argument (that is, $ {r}/{3}, {2r}/{3}$, and $ {4r}/{3}$), any other combination actually  yields a weaker estimate.

\begin{remark} It is interesting to observe that $C_{(X,d)}$ actually depends on the metric $d$, and is not invariant under homeomorphisms. We give first an example for    discrete spaces and afterwards a stronger result in the continuous case:

For the complete graph $K_3$, with the standard metric, it is easy to see that $C_{K_3}=3$ (see Proposition \ref{comgra} for the general case of $K_n$). If we now label the vertices   as $K_3=\{a,b,c\}$ and define the metric $d(a,c)=d(b,c)=2$ and $d(a,b)=1$, then taking the measure $\mu(a)=\mu(b)=1$ and $\mu(c)=2$, we get
$$
\frac{\mu(B(a,2r))}{\mu(B(a,r))}=\begin{cases}
 1,     & \text{if } 0<r\le1/2  \text{ or } r>2, \\
  2,    & \text{otherwise},
\end{cases}
$$
and similarly for $b$. At the vertex c:
 $$
\frac{\mu(B(c,2r))}{\mu(B(c,r))}=\begin{cases}
 1,     & \text{if } 0<r\le1  \text{ or } r>2, \\
  2,    & \text{otherwise}.
\end{cases}
$$
Thus, we find that $C_{(K_3,d)}\le C_\mu\le 2\neq3=C_{K_3}$ (in fact, using Proposition~\ref{p:isolated} we have that $C_{(K_3,d)}=2$).

In the continuous setting of $\mathbb R$   one can even show that there are metrics $d_1$ and $d_2$ for which $(\mathbb R, d_1)$ and $(\mathbb R,d_2)$ are homeomorphic but $C_{(\mathbb R,d_1)}<\infty$ and $C_{(\mathbb R,d_2)}=\infty$. In fact, taking $d_1$ to be the euclidean metric and
$$
d_2(x,y)=\frac{|x-y|}{1+|x-y|},
$$
we know that the two metrics are topologically equivalent. However, if for $0<r<1$ we consider the balls
$$
B_{d_2}(x,r)=\Big(x-\frac{r}{1-r},x+\frac{r}{1-r}\Big),
$$
it is clear that we can find $\bigg[\frac{\frac{r}{1-r}}{\frac{1/2}{1-1/2}}\bigg]=\big[\frac{r}{1-r}\big]$ disjoint balls of radius $1/2$ inside $B_{d_2}(0,r)$. Thus, using Proposition~\ref{p:lowerboundsr}, if $\mu$ were a doubling measure in $(\mathbb R,d_2)$ we get
$$
C_\mu\ge\Big[\frac{r}{1-r}\Big]^{1/\lceil2+\log_2 r\rceil}\underset{r\to1^-}{\longrightarrow}\infty,
$$
and hence $C_{(\mathbb R,d_2)}=\infty$, while $C_{(\mathbb R,d_1)}\le C_{|\cdot|}=2$, where $|\cdot|$ is the Lebesgue measure  (in fact, using Proposition~\ref{crn}, we do have the equality $C_{(\mathbb R,d_1)}=2$).
\medskip

Another natural example to consider here would be the ``snowflaking'' of $(\mathbb R, d_1)$: fix $0<\varepsilon<1$, and let $d_\varepsilon(x,y)=|x-y|^\varepsilon$. A similar argument as above yields that $C_{(\mathbb R, d_\varepsilon)}\geq 2^{1/\varepsilon}>2=C_{(\mathbb R, d_1)}$.
\end{remark}

\section{On the lower bound $C_{(X,d)}\geq2$}\label{sec3}

The purpose of this section is to prove the following

\begin{theorem}\label{Cgeq2}
If $(X,d)$ is a metric space with $|X|>1$, then $C_{(X,d)}\geq2$.
\end{theorem}

For the proof of this fact we will need first a series of preliminary results.

\begin{proposition}\label{p:isolated}
If $(X,d)$ is a metric space, with $|X|>1$, which has an isolated point, then $C_{(X,d)}\geq2$. In particular, this is the case on discrete or finite metric spaces.

\end{proposition}

\begin{proof}
Let $x\in X$ be an isolated point, and let $r_x=\sup\big\{r>0:B(x,r)=\{x\}\big\}$. Observe that $0<r_x<\infty$ and $B(x,r_x)=\{x\}$. Take $\varepsilon\in(0, 1/3)$. By definition of $r_x$, there is $y\in X$, with $r_x\leq d(x,y)< r_x(1+\varepsilon)$. Let now $\mu$ be a doubling measure on $(X,d)$. Since $B\big(y,\frac{r_x(1+\varepsilon)}{2}\big)\subset B(x,2r_x)\backslash\{x\}$, it follows that
$$
\mu\Big(B\Big(y,\frac{r_x(1+\varepsilon)}{2}\Big)\Big)\leq \mu(B(x,2r_x))-\mu(\{x\})\le(C_\mu-1)\mu(\{x\}).
$$
Now, since $B\big(y,\frac{r_x(1+\varepsilon)}{2}\big)\cup\{x\}\subset B(y,r_x(1+\varepsilon))$, we get
\begin{align*}
C_\mu\,\mu\Big(B\Big(y,\frac{r_x(1+\varepsilon)}{2}\Big)\Big)&\ge\mu(B(y,r_x(1+\varepsilon)))\geq\mu\Big(B\Big(y,\frac{r_x(1+\varepsilon)}{2}\Big)\Big)+\mu(\{x\})\\
&\ge\Big(1+\frac{1}{C_\mu-1}\Big)\mu\Big(B\Big(y,\frac{r_x(1+\varepsilon)}{2}\Big)\Big).
\end{align*}
Therefore, $C_\mu^2-C_\mu\ge C_\mu$ and hence $C_\mu\ge 2$.
\end{proof}

\medskip

As a complement to Proposition~\ref{p:isolated}, we will see next that if a metric space contains a line segment (which, in some sense, is the contrary to having an isolated point), then a different argument yields the same lower bound. Before, we need the following definition.

\begin{definition}
Let $m\in\mathbb N$, $r>0$, $\theta\in(0,1]$. An $(m,r,\theta)$-configuration is a finite collection of points $(x_i)_{i=0}^{\lceil \theta m \rceil}\subset X$ such that, for every $1\leq i\leq \lceil \theta m \rceil$, we have
\begin{equation}\label{distance_0i}
r\Big(1+\frac{1}{2m}\Big)\leq d(x_0,x_i)\leq r\Big(2-\frac{1}{2m}\Big),
\end{equation}
and for $i,j\in\{1,\ldots,\lceil \theta m \rceil\}$, with $i\neq j$,
\begin{equation}\label{distance_ij}
d(x_i,x_j)\geq \frac{r}{m}.
\end{equation}
\end{definition}

Given $\theta\in(0,1]$, we will say that a metric space $(X,d)$ contains arbitrarily long $\theta$-configurations if, for every $n\in\mathbb N$, there exist $m\geq n$, $r_m>0$ and an $(m,r_m,\theta)$-configuration $(x_i)_{i=0}^{\lceil \theta m \rceil}\subset X$.

\begin{lemma}\label{rootto2}
Let $\theta\in(0,1]$. For each $n\in\N$, the equation
\begin{equation}\label{polinx}
x^{n+5}-x^{n+4}-\theta2^{n-1}=0
\end{equation}
has a unique solution $x_n>1$, which satisfies $x_n\underset{n\rightarrow\infty}\longrightarrow2$.
\end{lemma}

\begin{proof}
Let us consider the function $f(x)=x^{n+5}-x^{n+4}-\theta2^{n-1}$. It is easy to check that $f$ is increasing and unbounded, for $x>1$, and since $f(1)<0$, we can define $x_n$ to be the only zero of $f$ in the set $(1,\infty)$. Let us denote $y_n={x_n}/2$, so that $\eqref{polinx}$ gives the equality
\begin{equation}\label{y_n}
2^5 y_n^{n+4}=\frac{\theta}{2y_n-1}.
\end{equation}
Since $f(2)>0$, it follows that $1/2<y_n<1$, for every $n$. We now prove that $y_n$ is monotone increasing with $n$. Indeed, suppose that there exists $n_0\in\mathbb N$, such that $y_{n_0}>y_{n_0+1}$. Then using \eqref{y_n}, we get
$$
\frac{y_{n_0}^{n_0+4}}{y_{{n_0}+1}^{{n_0}+5}}=\frac{2y_{{n_0}+1}-1}{2y_{n_0}-1},
$$
which yields
$$
1<\Big(\frac{y_{n_0}}{y_{{n_0}+1}}\Big)^{{n_0}+4}=y_{{n_0}+1}\frac{2y_{{n_0}+1}-1}{2y_{n_0}-1}<y_{{n_0}+1}<1,
$$
which is a contradiction. Therefore, the sequence is increasing and we find the limit  $\lim_{n\rightarrow\infty} y_n=L\in(1/2,1]$. But,  if $L<1$, then \eqref{y_n} would imply the equality
$$
\frac{\theta}{2L-1}=0,
$$
which is again contradiction. Thus, $L=1$ and this finishes the proof.
\end{proof}

\begin{theorem}\label{t:thetacon}
Let $(X,d)$ be a metric space which contains arbitrarily long $\theta$-con\-figura\-tions, for some $\theta\in(0,1]$. Then $C_{(X,d)}\geq2$.
\end{theorem}

\begin{proof}
Given $k\in\mathbb N$, let $m\geq k$ and $r_m>0$ such that $(x_i)_{i=0}^{\lceil \theta m \rceil}\subset X$ is an $(m,r_m,\theta)$-configuration. Let us see that the collection of balls  $\big\{B(x_i,\frac{r_m}{2m}):1\leq i\leq \lceil \theta m \rceil\big\}$ and $B(x_0,r_m)$  are pairwise disjoint. Indeed, if there is $z\in B(x_0,r_m)\cap B(x_i,\frac{r_m}{2m})$, then we would have
$$
d(x_0,x_i)\leq d(x_0,z)+d(z,x_i)<r_m\Big(1+\frac{1}{2m}\Big),
$$
which is impossible by \eqref{distance_0i}; similarly, if for $i,j\in\{1,\ldots, m\}$, with $i\neq j$, we can find  $z\in B(x_i,\frac{r_m}{2m})\cap B(x_j,\frac{r_m}{2m})$, then we would have $d(x_i,x_j)<r_m/m$, which is a contradiction with \eqref{distance_ij}. Moreover, we clearly have
$$
B\Big(x_i,\frac{r_m}{2m}\Big)\subset B(x_0,2r_m),
$$
for $1\leq i\leq \lceil \theta m \rceil$. Therefore, for any doubling measure $\mu$ on $(X,d)$, we have
\begin{equation}\label{measure_disjoint}
\mu(B(x_0,r_m))+\sum_{i=1}^{\lceil \theta m \rceil} \mu\Big(B\Big(x_i,\frac{r_m}{2m}\Big)\Big)\leq \mu(B(x_0,2r_m))\leq C_\mu\, \mu(B(x_0,r_m)).
\end{equation}
Observe that, for every $1\leq i\leq \lceil \theta m \rceil$, we also have
\begin{equation*}
B(x_0,r_m)\subseteq B\Big(x_i,(6m-1)\frac{r_m}{2m}\Big).
\end{equation*}
Indeed, if $y\in B(x_0,r_m)$, then
$$
d(y,x_i)\leq d(y,x_0)+d(x_0,x_i)<r_m+r_m\Big(2-\frac{1}{2m}\Big)=(6m-1)\frac{r_m}{2m}.
$$
Hence, since $\mu$ is a doubling measure, for every $1\leq i\leq \lceil \theta m \rceil$, using Lemma~\ref{lrs}  we have
\begin{equation}\label{doubling_0i}
\mu(B(x_0,r_m))\leq \mu\Big( B\Big(x_i,(6m-1)\frac{r_m}{2m}\Big)\Big)\leq C_\mu^{1+\log_2(6m-1)}\mu\Big(B\Big(x_i,\frac{r_m}{2m}\Big)\Big).
\end{equation}
Thus, putting  \eqref{measure_disjoint} and \eqref{doubling_0i} together, we get
\begin{equation}\label{Cmu}
1+\sum_{i=1}^{\lceil \theta m \rceil} C_\mu^{-1-\log_2(6m-1)}\leq C_\mu.
\end{equation}

Let $n\in \N$ be such that $2^{n-1}<m\leq 2^n$. Thus, for $1\leq i\leq m$ we have
\begin{equation}\label{log}
1+\log_2(6m-1)\leq 1+\log_2(6\cdot 2^{n}-1)< n+4.
\end{equation}
Now, \eqref{Cmu} and \eqref{log} yield
\begin{equation*}
C_\mu> 1+\theta2^{n-1}\frac{1}{C_\mu^{n+4}}.
\end{equation*}
Hence,
\begin{equation*}
C_\mu^{n+5}-C_\mu^{n+4}-\theta2^{n-1}>0.
\end{equation*}
Since $C_\mu\geq1$, we must then have $C_\mu> x_n$, where $x_n$ is the only solution, greater than 1, of the equation $x^{n+5}-x^{n+4}-\theta2^{n-1}=0$. Iterating this argument, for any $k\in\mathbb N$ and $2^n\geq m\geq k$, we can let $n\rightarrow \infty$ and, by Lemma~\ref{rootto2}, we obtain  $x_n\underset{n\rightarrow\infty}\longrightarrow2$. Therefore, $C_\mu\geq\sup_n x_n=2$ and thus $C_{(X,d)}\geq2$.
\end{proof}

\begin{proposition}
Let $(X,d,\mu)$ be a space of homogeneous type such that for every $\varepsilon>0$ there exist $x,y\in X$ such that
\begin{equation*}
\mu(B(x,d(x,y))\cap B(y,d(x,y))\leq\varepsilon\mu (B(x,2d(x,y))\cap B(y,2d(x,y)).
\end{equation*}
Then $C_\mu\geq2$.
\end{proposition}

\begin{proof}
Given $\varepsilon>0$, let $x,y\in X$ as in the hypothesis, and set $r=d(x,y)$. Since
$$
B(x,r)\cup B(y,r)\subset B(x,2r)\cap B(y,2r),
$$
it follows that
\begin{eqnarray*}
\mu(B(x,r))+\mu(B(y,r))&=&\mu(B(x,r)\cup B(y,r))+\mu(B(x,r)\cap B(y,r))\\
&\leq & (1+\varepsilon) \mu(B(x,2r)\cap B(y,2r)).
\end{eqnarray*}
Without loss of generality, let us assume $\mu(B(x,r))\leq \mu(B(y,r))$. Hence, we have
$$
C_\mu\geq \frac{\mu(B(x,2r))}{\mu(B(x,r))}\geq\frac2{1+\varepsilon}.
$$
As this holds for every $\varepsilon>0$, the conclusion follows.
\end{proof}

\begin{corollary}\label{cor:int}
Let $(X,d)$ be a metric space for which there exist $x,y\in X$ with $B(x,d(x,y))\cap B(y,d(x,y))=\emptyset$. Then, $C_{(X,d)}\geq 2$.
\end{corollary}

\begin{proposition}\label{compact}
If $(X,d)$ is a metric space such that, for some pair of distinct points $x_0,y_0\in X$, the closed ball 
$
\{z\in X:  d(x_0,z)\leq 2d(x_0,y_0)\}
$
is compact, then $C_{(X,d)}\geq2$.
\end{proposition}

\begin{proof}
Let $r_0=d(x_0,y_0)$ and set $K=\{z\in X:  d(x_0,z)\leq 2d(x_0,y_0)\}$ be a compact closed ball in $X$. Suppose first that for every $r\in(r_0/2,r_0)$ there exists $x(r)\in X$ with $d(x_0,x(r))=r$. In particular, given $m\in\mathbb N$, for every $i=1,\ldots,m-1$, we can define
$$
x^m_i=x\bigg(\frac{r_0}{2}\Big(1+\frac{2i+1}{2m}\Big)\bigg).
$$
Taking $x^m_0=x_0$, it follows that $(x^m_i)_{i=0}^{m-1}$ is an $(m-1,r_0/2,1)$-configuration. The conclusion now follows from Theorem~\ref{t:thetacon}.

Now, suppose there is $r\in (r_0/2,r_0)$ such that 
$$
\{z\in X:  d(x_0,z)=r\}=\emptyset.
$$
Let $A=\{z\in X: d(x_0,z)\leq r\}$ and $B=\{z\in X: r\leq d(x_0,z)\leq 2r_0\}$. Clearly, $A$ and $B$ are non-empty clopen subsets of $K$ such that $A\cap B=\emptyset$. Since $A$ and $B$ are actually compact, there exist $x_A\in A$ and $x_B\in B$ such that
$$
d(x_A,x_B)=\inf\{d(x,y):x\in A, \,y\in B\}\leq r_0.
$$
In this case we have
$$
B(x_A,d(x_A,x_B))\cap B(x_B,d(x_A,x_B))=\emptyset.
$$
Indeed, if $z\in B(x_A,d(x_A,x_B))\cap B(x_B,d(x_A,x_B))$, then 
$$
d(x_0,z)\leq d(x_0,x_A)+d(z,x_A)<d(x_0,x_A)+d(x_A,x_B)<2r_0.
$$
Therefore, either $z\in A$ or $z\in B$ which in either case yields a contradiction with the fact that $d(x_A,x_B)$ is minimal. Thus, in this case the conclusion follows from Corollary~\ref{cor:int}.
\end{proof}

It is worth to observe that we did not use the compactness property in the first part of the previous proof (the existence of the suitable configuration suffices). The following result is then an immediate consequence and shows that, with the reasonable hypothesis of completeness (which for a doubling metric space implies the existence of a doubling measure), the doubling constant is at least 2:

\begin{corollary} \label{completecons2}
If $(X,d)$ is a complete metric space with $|X|>1$, then $C_{(X,d)}\geq2$.
\end{corollary}

\begin{proof}
Take any two distinct points $x_0,y_0\in X$ and let us define, as before, the closed ball  $K=\{z\in X:  d(x_0,z)\leq 2d(x_0,y_0)\}$. Given, any doubling measure $\mu$, using Corollary~\ref{totallybounded}, we have that $K$ is totally bounded and, since $X$ is complete, $K$ is actually compact. The conclusion follows from Proposition~\ref{compact}.
\end{proof}

\begin{proof}[Proof of Theorem \ref{Cgeq2}]
Suppose $(X,d)$ is a (non necessarily complete) metric space with $|X|>1$ and a doubling measure $\mu$. Let $\tilde X$ denote the completion of $X$, and $\tilde \mu$ the extension of $\mu$ given by \cite[Lemma 1]{S}, which satisfies 
$$
C_{\tilde \mu}\leq C_{\mu}.
$$ 
By Corollary \ref{completecons2} we know that $C_{\tilde \mu}\geq2$, hence the conclusion follows.
\end{proof}

\section{More general lower bounds for $C_{(X,d)}$}\label{sec4}

In this section, we look for general estimates for $C_{(X,d)}$ under quite natural assumptions. Let us explore first the case when the measure of a ball essentially depends on its radius, and not on where its center is.  Recall that $(X,d,\mu)$ is called Ahlfors $Q$-regular, with constant $C\ge1$, if for every $x\in X$ and $0<r< \operatorname{diam}(X)$ we have
$$
\frac{1}{C} r^Q\leq \mu(B(x,r))\leq C r^Q.
$$
These    spaces play an important role in geometric measure theory (cf. \cite{HKST}). In particular, if $(X,d,\mu)$ is Ahlfors $Q$-regular, with constant $C\ge1$, then $C_\mu\geq 2^Q/C^2$. More generally, we have the following:

\begin{theorem}\label{homogeneous}
Let $(X,d,\mu)$ be a space of homogeneous type with $|X|>1$ and such that for some functions $\phi_1,\phi_2:\mathbb R_+\rightarrow \mathbb R_+$, every $x\in X$ and every $r>0$, we have
$$
\phi_1(r)\leq \mu(B(x,r))\leq \phi_2(r).
$$
If, in addition, $\phi_2$ is right continuous, then 
$$
C_\mu\geq2\sup_{r>0}\frac{\phi_1(r)}{\phi_2(r)}.
$$
\end{theorem}

\begin{proof}
Given $x,y\in X$ let 
$$
r(x,y)=\inf\{s>0:B(x,s)\cap B(y,s)\neq \emptyset\}.
$$
Clearly, we have
$$
\frac{d(x,y)}{2}\leq r(x,y)\leq d(x,y),
$$
and
$$
B(x,r(x,y))\cap B(y,r(x,y))=\emptyset.
$$
Now, for every $n\in\mathbb N$, we can take $z_n\in B(x,r(x,y)+ 1/n)\cap B(y,r(x,y)+\ 1/n)$. Note that, for each $n\in\mathbb N$, we have
$$
B(x,r(x,y))\cup B(y,r(x,y))\subset B(z_n,2r(x,y)+ 1/n).
$$
Therefore,
\begin{align*}
2\phi_1(r(x,y))&\leq\mu(B(x,r(x,y))\cup B(y,r(x,y)))\\
&\leq \mu\Big(B \Big(z_n,2r(x,y)+\frac1n \Big) \Big)\\
&\leq C_\mu \phi_2 \Big(r(x,y)+\frac{1}{2n} \Big).
\end{align*}
Since $\phi_2$ is right continuous, we have that $\phi \big(r(x,y)+\frac{1}{2n} \big)\underset{n\rightarrow\infty}{\longrightarrow} \phi(r(x,y))$. Thus, $C_\mu\geq2\phi_1(r(x,y))/\phi_2(r(x,y)),$ and the claim follows.
\end{proof}

\begin{theorem}\label{maxmin}
Let $(X,d)$ be a metric space, $\varepsilon_0>0$ and $\varphi:[0,\varepsilon_0)\rightarrow \mathbb R_+$ be an increasing continuous function such that for every $\varepsilon\in(0,\varepsilon_0)$ there exist $K>0$ and $N\in\mathbb N$ so that, for every $n\geq N$, we can find distinct points $(x_i)_{i=1}^n\subset X$ with
\begin{equation}\label{eq:maxmin}
\max_{i,j\in\{1,\ldots,n\}} d(x_i,x_j)\leq K n^{\varphi(\varepsilon)} \min_{i\neq j} d(x_i,x_j).
\end{equation}
Then $C_{(X,d)}\geq 2^{ {1}/{\varphi(0)}}$.
\end{theorem}

\begin{proof}
Given $\varepsilon\in(0,\varepsilon_0)$, let $K>0$, $N\in\mathbb N$ be as above, and for $n\geq N$,  let $(x_i)_{i=1}^n\subset X$ be distinct points satisfying \eqref{eq:maxmin}. Let $r=\min_{i\neq j} d(x_i,x_j)$ and $R=\max_{i,j\in\{1,\ldots,n\}} d(x_i,x_j)$. It follows that for every $i,j\in\{1,\ldots,n\}$ we have
$$
B(x_i, {r}/{2})\subset B(x_j,R+ {r}/{2}).
$$
Thus,
$$
\bigcup_{i=1}^n B(x_i, {r}/{2})\subset \bigcap_{j=1}^n B(x_j,R+ {r}/{2}).
$$
Let $\mu$ be a doubling measure on $(X,d)$, and let $i_0\in\{1,\ldots,n\}$ be such that
$$
\mu(B(x_{i_0}, {r}/{2}))=\min_{i\in\{1,\ldots,n\}}\mu(B(x_i, {r}/{2})).
$$
Since $B(x_i, {r}/{2})\cap B(x_{j}, {r}/{2})=\emptyset$, whenever $i\neq j$, it follows that
$$
n\mu(B(x_{i_0}, {r}/{2}))\leq \mu\Big(\bigcup_{i=1}^n B(x_i, {r}/{2})\Big)\leq \mu\Big(\bigcap_{j=1}^n B(x_j,R+ {r}/{2})\Big)\leq \mu (B(x_{i_0},R+ {r}/{2})).
$$
Therefore, by Lemma \ref{lrs} and \eqref{eq:maxmin}, it follows that
$$
n\leq \frac{\mu (B(x_{i_0},R+ {r}/{2}))}{\mu(B(x_{i_0}, {r}/{2}))}\leq C_\mu^{\lceil\log_2(\frac{R+ {r}/{2}}{ {r}/{2}})\rceil}\leq C_\mu^{\lceil\log_2(2Kn^{\varphi(\varepsilon)}+1)\rceil}.
$$
Thus
$$
\log_2 C_\mu\geq\frac{\log_2 n}{\lceil\log_2(2Kn^{\varphi(\varepsilon)}+1)\rceil}\geq \frac{\log_2 n}{\log_2(2K+1)+\varphi(\varepsilon)\log_2 (n+1)}\underset{n\rightarrow\infty}\longrightarrow\frac{1}{\varphi(\varepsilon)}.
$$
Hence, $C_\mu\geq 2^{ {1}/{\varphi(\varepsilon)}}$ and letting $\varepsilon\rightarrow 0$ we get 
$C_\mu\geq 2^{ {1}/{\varphi(0)}}.$
\end{proof}

\begin{remark}
The previous result can be used to provide an alternative argument for Theorem~\ref{t:thetacon}. Indeed, suppose that for any $k\in\mathbb N$, there exist $m\geq k$ and $r_m>0$ such that $(x_i)_{i=0}^{\lceil \theta m \rceil}\subset X$ form an $(m,r_m,\theta)$-configuration. It is easy to check that if we set $n=\lceil \theta m \rceil$, then
$$
\max_{i,j\leq n} d(x_i,x_j)\leq 2 r_m\leq \frac{2}{\theta}{n} \min_{i\neq j} d(x_i,x_j).
$$
Hence, Theorem \ref{maxmin} with $\varphi(t)=1$, for every $t>0$, yields that $C_{(X,d)}\geq 2$.
\end{remark}

Let $K_n$ denote the complete graph with $n$ vertices. We will say that $(X,d)$ contains a copy of $K_n$ if there exist $r>0$ and $(x_i)_{i=1}^n\subset X$ such that $d(x_i,x_j)=r$, for every $i\neq j$. The following result is an extension of Proposition~\ref{goldenratio}, since $(X,d)$ always contains a copy of $K_2$.

\begin{theorem}
If $(X,d)$ contains a copy of $K_n$ for some $n\geq2$, then $C_{(X,d)}\geq\frac{1+\sqrt{4n-3}}{2}$.
\end{theorem}

\begin{proof}
Let $\mu$ be any doubling measure on $(X,d)$. Let $(x_i)_{i=1}^n\subset X$ be such that $d(x_i,x_j)=r$, for every $i\neq j$. Fix $\lambda>0$. Suppose first, that for some $i\in\{1,\ldots,n\}$ we have
$$
\mu(B(x_i, {2r}/{3}))<\lambda\sum_{j\neq i} \mu(B(x_j, {r}/{3})).
$$
Since for different  $j,j'\in\{1,\ldots,n\}\backslash\{i\}$ we have
\medskip

\begin{enumerate}[(i)]
\item $B(x_i, {2r}/{3})\cap B(x_j, {r}/{3})=\emptyset$,
\item $B(x_j, {r}/{3})\cap B(x_{j'}, {r}/{3})=\emptyset$,
\item $B(x_j, {r}/{3})\subset B(x_i, {4r}/{3})$,
\end{enumerate}
\medskip

\noindent
it follows that
\begin{align*}
C_\mu \mu(B(x_i, {2r}/{3}))&\geq \mu(B(x_i, {4r}/{3}))\geq  \mu(B(x_i, {2r}/{3}))+\sum_{j\neq i} \mu(B(x_j, {r}/{3}))\\
&>(1+ 1/\lambda)  \mu(B(x_i, {2r}/{3})).
\end{align*}
Thus, in this case we have $C_\mu>1+ {1}/{\lambda}$.

Now, suppose that for every $i\in\{1,\ldots,n\}$ we have
$$
\mu(B(x_i, {2r}/{3}))\geq\lambda\sum_{j\neq i} \mu(B(x_j, {r}/{3})).
$$
Taking the sum over all $i\in\{1,\ldots,n\}$, we get
$$
C_\mu\sum_{i=1}^n \mu(B(x_i, {r}/{3}))\geq \sum_{i=1}^n \mu(B(x_i, {2r}/{3})) \geq\sum_{i=1}^n\lambda\sum_{j\neq i} \mu(B(x_j, {r}/{3})).
$$
Thus, in this case it follows that $C_\mu\geq\lambda(n-1)$. Hence, we have proved that 
$$
C_\mu\geq \sup_{\lambda>0}\min\{1+1/\lambda,\lambda(n-1)\}.
$$
Optimizing in $\lambda>0$, the conclusion follows.
\end{proof}

\section{Examples}\label{sec5}

In this Section we are going to find some explicit values of the constants $C_{(X,d)}$ for a wide range of metric spaces $(X,d)$.  We start with the case of the finite dimensional real spaces with any of the equivalent $\ell^p$ metrics, $1\le p\le\infty$.
\begin{proposition}\label{crn} For very $n\in\mathbb N$ and $1\le p\le\infty$, we have that
$C_{(\mathbb R^n,\|\cdot\|_p)}=2^n$.
\end{proposition}

\begin{proof}
If $\lambda$ denotes the Lebesgue measure we clearly have $C_{(\mathbb R^n,\|\cdot\|_p)}\leq C_\lambda=2^n$. Conversely, given $1\le p\le \infty$, it is an easy geometric fact to observe that there exists a constant $c_p>0$ such that, for every $k\in\mathbb N$, we can find points $x_{j,k}\in\mathbb R^n$, $j=1,\dots,k^n$ satisfying that the collection of balls $\big\{B_p(x_{j,k},c_p/k)\big\}_{j=1,\dots,k^n}$ are pairwise disjoint and $B_p(x_{j,k},c_p/k)\subset B_p(0,1)$, where $B_p$ is a ball with respect to the metric $\|\cdot\|_p$.

\medskip

Hence, Proposition \ref{p:lowerboundsr} yields that, for every doubling measure $\mu$ on $(\mathbb R^n,\|\cdot\|_p)$, it holds that
$$
C_\mu\geq k^{ \frac{n}{1+\log_2(2k/c_p)}}\underset{k\to\infty}{\longrightarrow}2^n.
$$
\end{proof}

We consider now the case of a simple and connected graph $G$, as a metric space endowed with the shortest path distance $d$. We will use the standard notation on $G$: for finite graphs,  $n$ is the number of vertices $V(G)$ and $m$ the cardinality of its edges $E(G)$; for a vertex $v\in V(G)$, $d(v)$ is the degree (the number of neighbors or, equivalently, the cardinality of the sphere $S(v,1)=\{u\in V(G):d(v,u)=1\}$); $\Delta$ is the maximum degree of $G$. For short, we will denote $C_G=C_{(G,d)}$ for the least doubling constant with $G$ equipped with the shortest path distance $d$.

In general, it is not true that the cardinality measure is always doubling on $G$. A necessary (but not sufficient) condition for this to happen is that $\Delta<\infty$. Moreover, there are  (infinite) graphs $G$ where no doubling measure exists (i.e., $C_G=\infty$), even though $G$ is always a complete metric space. For example, it is easy to see that this is the case for the $k$-homogeneous tree $T_k$, $k\ge3$, since it is not doubling in the metric sense \cite{SorTra}.

If $G$ is finite, then for the cardinality measure $\lambda$   we have that, for every $x\in V(G)$ and $r>0$
$$
\frac{\lambda(B(x,2r))}{\lambda(B(x,r))}\le\frac n1.
$$
This inequality, together with Proposition~\ref{p:isolated}, gives us that if $n=|V(G)|\ge2$, then $2\le C_G\le n$. It goes without saying that on finite graphs, all measures are doubling (always assuming the restriction that balls should have positive measure).

\begin{proposition}\label{comgra}
Let $K_n$ denote the complete graph with $n$ vertices. Then $C_{K_n}=n$.
\end{proposition}

\begin{proof}
Let $V=\{x_j:1\leq j\leq n\}$ denote the set of vertices of $K_n$. Given a measure $\mu$ on $K_n$ (which is trivially doubling), let $a_j=\mu(\{x_j\}).$ Since $\mu(\{x_j\})=\mu(B(x_j,r))$, for any $0<r\leq1$, we have that $a_j>0$, for every $1\leq j\leq n$. Let $a_k=\min\{a_j:1\leq j\leq n\}$. Then
$$
C_\mu\geq\frac{\mu(B(x_k, {3}/{2}))}{\mu(B(x_k, {3}/{4}))}=\frac{\sum_{j=1}^n a_j}{a_k}\geq n.
$$

Thus,  $n \ge C_{K_n}=\inf{C_\mu}\ge n$.
\end{proof}

\begin{proposition}
Let $S_n$ denote the star graph with $n$ vertices; that is, one vertex of degree $n-1$ and $n-1$ vertices of degree $1$.  Then $C_{S_n}=1+\sqrt{n-1}$.
\end{proposition}

\begin{proof}
Let $V=\{x_j:1\leq j\leq n\}$ denote the set of vertices of $S_n$, with $x_1$ being the vertex of degree $n-1$. Given a doubling measure $\mu$ on $S_n$, let $a_j=\mu(\{x_j\})>0$ and $a=\sum_{j=1}^n a_j$. We have that
$$
\mu(B(x_1,r))=\left\{
\begin{array}{lll}
a_1,  &   &  r\leq1, \\
a,  &   &  r>1,
\end{array}
\right.\quad\text{and}\quad
\mu(B(x_j,r))=\left\{
\begin{array}{lll}
a_j,  &   &  r\leq1, \\
a_j+a_1,  &   & 1<r\le 2,\\
a, & & r>2,
\end{array}
\right.
$$
for $j\neq1$. Therefore,  
$$
\sup_{r>0}\frac{\mu(B(x_1,2r))}{\mu(B(x_1,r))}=\max\Big\{1,\frac{a}{a_1}\Big\}=\frac{a}{a_1},
$$
while for $j\neq1$ we get
$$
\sup_{r>0}\frac{\mu(B(x_j,2r))}{\mu(B(x_j,r))}=\max\Big\{1,\frac{a_j+ a_1}{a_j},\frac{a}{a_j+a_1}\Big\}.
$$
Since $\frac{a}{a_j+a_1}<\frac{a}{a_1}$, we have that
$$
C_\mu=\max\Big\{\frac{a}{a_1},\frac{a_j+ a_1}{a_j}\Big\}.
$$
Now, if $a_1=\min\{a_j:1\leq j\leq n\}$, then
\begin{equation*}
C_\mu=\frac{a}{a_1}\geq n.
\end{equation*}
Otherwise, suppose $a_{j_0}=\min\{a_j:1\leq j\leq n\}<a_1$. Then we want to compute
$$
A=\inf\Big\{\max\Big\{\frac{a}{a_1},\frac{a_{j_0}+ a_1}{a_{j_0}}\Big\}:0<a_{j_0}<a_1\Big\}.
$$
If we set $r=\frac{a_1}{a_{j_0}}$ and $s=\frac{\sum_{j\neq1,j_0} a_j}{a_{j_0}}$, then we get
$$
A\geq\inf\Big\{\max\Big\{\frac{s+1}{r}+1,r+1\Big\}:r>1,s\geq n-2\Big\}.
$$
Note that $\frac{s+1}{r}+1>r+1$ if and only if $r<\sqrt{s+1}$. Thus,
\begin{align*}
A&\geq\min\Big\{\inf\Big\{\frac{s+1}{r}+1:1<r<\sqrt{s+1},s\geq n-2\Big\},\\
&\qquad\qquad\inf\big\{r+1:r>\sqrt{s+1},s\geq n-2\big\}\Big\}\\
&\geq 1+\sqrt{n-1}.
\end{align*}
Hence, we get
$
C_{S_n}\geq 1+\sqrt{n-1}.
$
\medskip

Conversely, if we consider in $V$ the measure $\mu$ given by $\mu(\{x_1\})=\sqrt{n-1}$ and $\mu(\{x_j\})=~1$, for $j\neq 1$,  we finally get
$$
C_{S_n}\leq C_\mu=\max\Big\{\frac{n-1+\sqrt{n-1}}{\sqrt{n-1}},\frac{\sqrt{n-1}+1}{1}\Big\}=\sqrt{n-1}+1.
$$
This finishes the proof.
\end{proof}

\begin{proposition}
For $n\geq3$, let $C_n$ denote the $n$-cycle graph; that is, a connected graph of $n$ vertices all of them with degree 2. Then, $C_{C_n}=3$.
\end{proposition}

\begin{proof}
Let $V=\{x_j:1\leq j\leq n\}$ denote the set of ordered vertices of $C_n$. Given any   measure $\mu$ on $C_n$, let $a_j=\mu(\{x_j\})>0$. Let $1\le j_0\leq n$ such that $a_{j_0}=\min\{ a_j: 1\leq j\leq n\}$. Hence, we have that
$$
C_\mu\geq \frac{\mu(B(x_{j_0},2))}{\mu(B(x_{j_0},1))}=\frac{a_{j_0-1}+a_{j_0}+a_{j_0+1}}{a_{j_0}}\geq 3,
$$
(we understand that $j_0-1=n$, if $j_0=1$, and $j_0+1=1$, if $j_0=n$). Since this holds for any (doubling) measure in $C_n$, it follows that $C_{C_n}\geq 3$.

For the converse, let $\mu_{\#}$ be the counting the measure in $C_n$; that is,  $\mu_{\#}(\{x_j\})=1$, for $1\leq j\leq n$. We first observe that, on any graph, $B(x,r)=B(x,\lceil r\rceil)$ and hence we only need to consider values of $r>0$ for which $r\in\mathbb N$ or $2r\in\mathbb N$. Moreover, since $\mu_{\#}(B(x,r))=\min\{2r-1,n\}$, $r\in\mathbb N$, then we can easily restrict the radius to the range  $1/2\le r\le (n+1)/4$. The important remark for $C_n$ is that if $2r\in\mathbb N$, but $r\notin\mathbb N$ (e.g., $r=1/2,3/2,5/2,...$), then $B(x,r)=B(x,r+1/2)$ and hence we obtain
$$
C_{\mu_{\#}}\ge\frac{\mu_{\#}(B(x,2r+1))}{\mu_{\#}(B(x,r+1/2))}\ge\frac{\mu_{\#}(B(x,2r))}{\mu_{\#}(B(x,r))},
$$
showing that we can further reduce the radius of the balls  to the simpler condition $r\in\{1,2,\dots,\lfloor (n+1)/4\rfloor\}$. Finally, for those values of $r$:
$$
\frac{\mu_{\#}(B(x,2r))}{\mu_{\#}(B(x,r))}=\frac{4r-1}{2r-1}\le 3.
$$
Therefore, $C_{\mu_{\#}}=3$ and this finishes the proof.
\end{proof}

\end{document}